\documentclass[11pt,a4paper]{amsart}
\usepackage{amsfonts}
\usepackage{amsthm}
\usepackage{amsmath}
\usepackage{amscd}
\usepackage[latin2]{inputenc}
\usepackage{t1enc}
\usepackage[mathscr]{eucal}
\usepackage{indentfirst}
\usepackage{graphicx}
\usepackage{graphics}
\usepackage{pict2e}
\usepackage{epic}
\numberwithin{equation}{section}
\usepackage[margin=2.9cm]{geometry}
\usepackage{epstopdf} 
\usepackage{siunitx}

\theoremstyle{plain}
\newcommand{\divides}{\mid}

\DeclareMathOperator{\dime}{dim}
\DeclareMathOperator{\pideg}{PI-deg}

\DeclareMathOperator{\gcdi}{gcd}

\DeclareMathOperator{\diagonal}{diag}

\DeclareMathOperator{\tors}{tor}

\DeclareMathOperator{\kere}{ker}
\DeclareMathOperator{\ran}{rank}

\DeclareMathOperator{\mspect}{Maxspec}
\usepackage{mathtools}

\numberwithin{equation}{section}
 
\newtheorem{theo}{Theorem}[section]

\newtheorem{defi}{Definition}[section]
\newtheorem{lemma}{Lemma}[section]
\newtheorem{remark}{Remark}[section]

\newtheorem{coro}{Corollary}[subsection]
\newtheorem{prop}{Proposition}[section]

\newtheorem*{Theorem D}{Theorem D}
\newtheorem*{Theorem C}{Theorem C}
\newtheorem*{Theorem A}{Theorem A}
\newtheorem*{Theorem B}{Theorem B}

\begin{document}

\title{Azumaya Loci Of The Quantum Euclidean $2n$-Space.}

\author[Snehashis Mukherjee]{
Snehashis Mukherjee}

\address{Ramakrishna Mission Vivekananda Educational and Research Institute\\ School of Mathematical Science\\Belur, Howrah, India, 711202.}

\email{tutunsnehashis@gmail.com}

\begin{abstract} This article delves into the Azumaya loci of quantum euclidean 2n-space, providing a comprehensive exploration. We introduce a significant class of maximal-dimensional simple modules associated with this algebra. Moreover, we establish a proof asserting that the maximal-dimensional simple modules are exactly the ones we have constructed. The computation of the algebra's center, coupled with the necessary and sufficient conditions derived for maximal-dimensional simple modules, enables us to determine the Azumaya locus of the algebra.
\end{abstract}

\maketitle
\section{Introduction}	
A fundamental
invariant of a noncommutative prime polynomial identity ring $R$ is its Azumaya locus $A(R) \subset \text{Maxspec} \ Z(R)$,
the dense open subset of Maxspec $Z(R)$ which parametrizes the maximal dimensional simple $R$-modules.  More precisely, $\mathbf{m} \in A(R)$ if and only if $\mathbf{m}R$ is the annihilator in
$R$ of a simple $R$-module V with dim$_\mathbb{K}V = n$, where $n$ is the PI-degree of $ R$, the
maximal $\mathbb{K}$-dimension of simple $R$-modules.\\
Let $\mathbb{K}$ be a field and let $\mathbb{K}^*$ denotes the group $\mathbb{K}\setminus\{0\}$. Let $q \in \mathbb{K}^*$. The coordinate ring $\mathcal{O}_{q}(\mathfrak{o}\mathbb{K}^{2n})$ of the quantum euclidean $2n$-space is the algebra generated over the field $\mathbb{K}$ by the variables $x_1,\cdots ,x_n,y_1, \cdots, y_n$ subject to the relations
\begin{align*}
x_ix_j&=qx_jx_i \ \ \ \ \ \ \ \ \ \ \ \ \forall\ \ \ i<j. \\
    y_iy_j&=q^{-1}y_jy_i \ \ \ \ \ \ \ \ \ \ \ \ \forall\ \ \ i<j.\\
x_iy_j&=q^{-1}y_jx_i \ \ \ \ \ \ \ \ \ \ \ \ \forall \ \ \ i\neq j.\\
x_iy_i&=y_ix_i+\sum_{l<i}(1-q^{-2})y_lx_l \ \ \ \ \forall i.
\end{align*} 
The algebra $\mathcal{O}_{q}(\mathfrak{o}\mathbb{K}^{N})$ first arose in \cite{fa} and was given a simpler set of relations in \cite{mu}. The generators for the even case with $N=2n$ are given by $X_1,X_{1'}, \cdots, X_n,X_{n'}$. To get the above relations set $x_i=X_{(n+1-i)'}$ and $y_i=q^{(n+1-i)}X_{(n+1-i)}$.
\par Musson studied ring-theoretic properties of the coordinate rings of quantum euclidean space in \cite{mu}. Primitive ideals in the coordinate ring of quantum euclidean space was studied by Sei Qwon Oh and Chun Gil Park in \cite{oo}. Oh also wrote an article on quantum and poisson structures of multi-parameter symplectic and euclidean spaces \cite{o}. All these papers consider the generic case. Although different aspects of quantum euclidean space have been studied in all these years its non-generic version has not been studied in details yet.\\ 
It is noteworthy that recent investigations by Brown and Yakimov \cite{bry} have provided valuable insights into the correlation between the Azumaya locus $\mathcal{AL}(R)$ and the discriminant ideal $D_{l}$ of a prime PI algebra $R$. Their main theorem \cite[Main Theorem]{bry} establishes, under mild assumptions, the alignment of the zero set of the discriminant ideal $D_l(Z/Z(R),tr)$ of a prime affine polynomial identity (PI) algebra $R$ with the zero set of the modified discriminant ideal $MD_l(Z/Z(R),tr)$ of $R$. They offer an explicit description of this set in terms of the dimensions of the irreducible representations of $R$.

Additionally, Brown and Yakimov prove that when $l$ is the square of the PI-degree of $R$, this zero set precisely complements the Azumaya locus of $R$. They leverage this description to classify the Azumaya loci of the multiparameter quantized Weyl algebras at roots of unity.

\par The purpose of this paper is to give a complete description of the Azumaya loci of the quantum euclidean $2n$-space at roots of unity. We compute the center of quantum euclidean $2n$-space at roots of unity, give a  necessary and sufficient condition for obtaining maximal dimensional simple modules and then describe the Azumaya loci.

Since Azumaya locus corresponds to maximal dimensional simple modules, in this article we will focus on constructing maximal dimensional simple modules over this algebra for non generic case. For an infinite-dimensional non-commutative algebra, the classification of its simple modules poses a formidable challenge in general. It is noteworthy that, despite the significance of studying the representation theory of quantum algebras, no comprehensive method has been established to date.
\par The existing examples of successfully classified algebras primarily revolve around generalized Weyl algebras, as seen in works such as \cite{ba1}, \cite{ba2}, \cite{ba3}, \cite{ba5}, \cite{ba6}, \cite{ba7}, \cite{ba8}, and \cite{ba9}. Additionally, Ore extensions with Dedekind rings as coefficient rings have been explored in \cite{ba4}.

The classification of finite-dimensional simple $R$-modules for a specific class of iterated skew polynomial rings, denoted as $R=A[y;\alpha][x;\alpha^{-1},\delta]$, where $A$ is an affine domain over an algebraically closed field, was achieved by Jordan in \cite{dj}. In this article, we present a novel approach to constructing simple modules over $\mathcal{O}_{q}(\mathfrak{o}\mathbb{K}^{2n})$. \\
  In the present article our main focus is on the study of the following problem: To prove the necessary and sufficient condition for obtaining maximal dimensional simple modules over $\mathcal{O}_{q}(\mathfrak{o}\mathbb{K}^{2n})$ for non generic case.\\
 Let us define for each $i, \ 1 \leq i \leq n, \ \omega_i:=\sum_{l \leq i}(1-q^{-2})y_lx_l$. These elements are going to play a crucial role throughout the article.
\begin{remark} \cite[Lemma 2.1]{aj}
Note that for all $1 \leq i \leq n$, $\omega_i$ is a normal element of  $\mathcal{O}_{q}(\mathfrak{o}\mathbb{K}^{2n})$. More precisely, we have
\begin{itemize}
    \item [1.] For all $i,j$ with $ 1 \leq i < j \leq n$, $\omega_ix_j=q^2x_j\omega_i$ and $\omega_iy_j=q^{-2}y_j\omega_i$.
    \item[2.] For all $i, \ 1 \leq j  \leq i \leq n$, $\omega_ix_j=x_j\omega_i$ and $\omega_iy_j=y_j\omega_i$.
    \item[3.] For all $i,j$ with $1 \leq i,j \leq n$, $\omega_i\omega_j=\omega_j\omega_i$.
\end{itemize}
\end{remark}
Let $A$ be an algebra and $M$ be a right $A$-module and $S\subset A$ be a right Ore set. The submodule
$$\tors_{S}(M):=\{m\in M:ms=0\ \text{for some}\  s\in S\}\ $$
is called the $S$-torsion submodule of $M$. The module $M$ is said to be $S$-torsion if tor$_{S}(M)=M$ and $S$-torsion-free if tor$_{S}(M)=0$. If the Ore set $S$ is generated by $x\in A$, we simply say that the $S$-torsion/torsion-free module $M$ is $x$-torsion/torsion-free.
\par A non zero element $x$ of an algebra ${A}$ is called a normal element if $x{A}={A}x$. Clearly if $x$ is a normal element of $A$, then the set $\{x^i:~i\geq 0\}$ is an Ore set generated by $x$. The next lemma is obvious.
\begin{lemma}\label{itn}
Suppose that $A$ is an algebra, $x$ is a normal element of $A$ and $M$ is a simple $A$-module. Then either $Mx=0$ (if $M$ is $x$-torsion) or the map $x_{M}:M\rightarrow M, m\mapsto mx$ is an isomorphism (if $M$ is $x$-torsion-free).
\end{lemma}
The above lemma says that the action of a normal element on a simple module is either trivial or invertible.
So each $\omega_i$, $x_1,y_1$ are normal in $\mathcal{O}_{q}(\mathfrak{o}\mathbb{K}^{2n})$. Let $\Omega$ be the multiplicative set generated by $\omega _i$ for all $i=1,\cdots,n-1$. So for any simple module $N$ over $\mathcal{O}_{q}(\mathfrak{o}\mathbb{K}^{2n})$, 
 \begin{center}
    $ \tors_\Omega(N)$=\{$n \in N : n\omega =0$ for some $\omega \in \Omega\}$
 \end{center}
 is either trivial or $N$.
\par Throughout the paper a module means a right module, $q$ is a primitive $m$-th root of unity,  $m$ being odd and $\mathbb{K}$ is an algebraically closed field. Also, $+$ denotes the addition in $\mathbb{Z}/m\mathbb{Z}$.
\par The arrangement of the paper is as follows:
Section $2$ provides an introduction to Polynomial Identity algebras, with a focus on establishing the main properties. Theorem $\ref{finite}$ asserts that $\mathcal{O}_{q}(\mathfrak{o}\mathbb{K}^{2n})$ qualifies as a prime PI algebra, accompanied by a discussion on Hanyal's findings concerning its PI degree.

Moving on to Section $3$, the PI degree computation of the prime factor of $\mathcal{O}_{q}(\mathfrak{o}\mathbb{K}^{2n})$ is presented. Notably, we present the subsequent result:
\begin{Theorem A}(Theorem \ref{main})
Under the roots of unity assumption for any $p$ with $2 \leq p \leq n-1$, the factor algebra $\mathcal{O}_q(\mathfrak{o}\mathbb{K}^{2n})/\langle \omega_p \rangle$ is a prime PI ring and $\pideg\mathcal{O}_q(\mathfrak{o}\mathbb{K}^{2n})/\langle \omega_p \rangle = m^{n-2}$.
\end{Theorem A}
Section $4$ provides a comprehensive exposition on the explicit structure of all potential maximal dimensional simple modules over $\mathcal{O}_{q}(\mathfrak{o}\mathbb{K}^{2n})$. Additionally, we establish that the set of maximal dimensional simple modules introduced in Section $4$ forms the entirety of such modules. The key finding in this section is encapsulated in the following theorem:
\begin{Theorem B}(Theorem \ref{maini})
 $N$ is a maximal dimensional simple $\mathcal{O}_q(\mathfrak{o}\mathbb{K}^{2n})$-module if and only if $N$ is $\omega_2,\cdots,\omega_{n-1}$ torsion-free simple $\mathcal{O}_{q}(\mathfrak{o}\mathbb{K}^{2n})$-module.
 \end{Theorem B}
In Section $5$, the focus is on determining the center of $\mathcal{O}_{q}(\mathfrak{o}\mathbb{K}^{2n})$, and the following theorem is established:
\begin{Theorem C}(Theorem \ref{cen})
The center $Z_n$ is the subalgebra generated by $x_1^ay_1^b$ where $a,b \in \{1,\cdots,m-1\}$ such that $a+b=m$, $x_i^m,y_i^m,i=1,\cdots,m$.
\end{Theorem C}
Additionally, an essential result (Lemma $\ref{vimp}$) is proven, involving the normal elements $\omega_i$ in $\mathcal{O}_{q}(\mathfrak{o}\mathbb{K}^{2n})$. Leveraging Lemma $\ref{vimp}$ and Theorem $\ref{cen}$, along with the necessary and sufficient condition derived for maximal dimensional simple modules in Section $4$, Section $5$ culminates in determining the Azumaya loci of quantum euclidean $2n$-space. The conclusive result is encapsulated in the following theorem:
\begin{Theorem D}
Let $\Psi_{\mathbf{m}}:Z_n \rightarrow \mathbb{K}$ denote the central character of $\mathbf{m} \in\mspect ~Z_n\left(\mathcal{O}_{q}(\mathfrak{o}\mathbb{K}^{2n})\right)$. Suppose that $\Psi_{\mathbf{m}}(x_i^m)=\alpha_i$ and $\Psi_{\mathbf{m}}(y_i^m)=\beta_i$. Then the
Azumaya locus of $\mathcal{O}_{q}(\mathfrak{o}\mathbb{K}^{2n})$ is given by
\begin{align*}
\mathcal{AL}(\mathfrak{o}\mathbb{K}^{2n})&=\{ \mathbf{m} \in  \text{Maxspec}~Z_n\left(\mathcal{O}_{q}(\mathfrak{o}\mathbb{K}^{2n})\right) :\Psi_{\mathbf{m}}(\omega_i^m) \neq 0\}\\ 
&=\{ \mathbf{m} \in \text{ Maxspec}~Z_n(\mathcal{O}_{q}(\mathfrak{o}\mathbb{K}^{2n})) : \sum_{l \leq i}\beta_l \alpha_l\neq 0,\ \ 2\leq i\leq n-1\}.
\end{align*}
\end{Theorem D}

\section{Preliminaries}

In this section, we recall relevant information about the Polynomial Identity algebra and we prove that $\mathcal{O}_q(\mathfrak{o}\mathbb{K}^{2n})$ is a PI algebra which enables us to comment on the $\mathbb{K}$-dimension of simple modules over $\mathcal{O}_q(\mathfrak{o}\mathbb{K}^{2n})$ at roots of unity.
\par A ring $R$ is said to be a Polynomial Identity (PI) ring if $R$ satisfies some monic polynomial $f\in \mathbb{Z}\langle x_1,\ldots,x_k\rangle$ i.e., $f(r_1,\ldots,r_k)=0$ for all $r_i\in R$.  The minimal degree of a PI ring $R$ is the least degree of all monic polynomial identities for $R$. PI rings cover a large class of rings including commutative rings. Commutative rings satisfy the polynomial identity $x_1x_2-x_2x_1$ and therefore have minimal degree $2$. Let us recall a result which provides a sufficient condition for a ring to be PI. 
\begin{prop}\emph{(\cite[Corollary 13.1.13]{mcr})}{\label{f}}
If $R$ is a ring that is a finitely generated module over a commutative subring, then $R$ is a PI ring.
\end{prop}
We now define the PI degree of prime PI-algebras. This definition will suffice because the algebra covered in this article is prime. We now recall one of the fundamental results from the Polynomial Identity theory.
\begin{theo}[Posner's Theorem {\cite[Theorem 13.6.5]{mcr}}]\label{pos}
 Let $A$ be a prime PI ring with centre $Z(A)$ and minimal
degree $d$. Let $S=Z(A)\setminus\{0\}$, $Q=AS^{-1}$ and $F=Z(A)S^{-1}$. Then $Q$ is a central simple algebra with centre $F$ and $\dime_{F}(Q)=(\frac{d}{2})^2$.
\end{theo} 
\begin{defi}
The PI degree of a prime PI ring $A$ with minimal degree $d$ is $\pideg(A)=\frac{d}{2}$.
\end{defi}
The following result provides an important link between the PI degree of a prime affine PI algebra over an algebraically closed field and the $\mathbb{K}$-dimension of its irreducible representations (cf. \cite[Theorem I.13.5, Lemma III.1.2]{brg}). 
\begin{prop}\label{sim}
Let $A$ be a prime affine PI algebra over an algebraically closed field $\mathbb{K}$, with PI-deg($A$) = $n$ and $V$ be a simple $A$-module. Then $V$ is a vector space over $\mathbb{K}$ of dimension $t$, where $t \leq n$, and $A/ann_A(V) \cong M_t(\mathbb{K})$. Moreover, the upper bound PI-deg($A$) is attained by some simple $A$-modules.
\end{prop}
Now in the root of unity setting, the above proposition yields the following result.
\begin{theo} \label{finite}
 $\mathcal{O}_{q}(\mathfrak{o}\mathbb{K}^{2n})$ is a PI ring.
\end{theo}
\begin{proof}
Consider the polynomial algebra $P= \mathbb{K}[x_i^m,y_i^m: 1\leq i\leq n]$. Note $\mathcal{O}_{q}(\mathfrak{o}\mathbb{K}^{2n})$ becomes a finitely generated module over $P$. Hence the assertion follows from proposition \cite[Corollary 13.1.13]{mcr}.
\end{proof}
From Theorem $\ref{finite}$ and \cite[Theorem I.13.5]{brg}, we conclude if $M$ is a $\mathcal{O}_{q}(\mathfrak{o}\mathbb{K}^{2n})$-simple module, then $\dime_{\mathbb{K}}M \leq \pideg(\mathcal{O}_{q}(\mathfrak{o}\mathbb{K}^{2n}))$.
\par The following few results will be useful.
\begin{prop}\emph{\cite[Proposition 7.1]{di}}\label{quan}
Let $\mathbf{q}=\left(q_{ij}\right)$ be an  $n \times n$ multiplicatively antisymmetric matrix over $\mathbb{K}$.
 Suppose that $q_{ij}=q^{h_{ij}}$ for all $i,j$, where $q \in \mathbb{K}^*$ is a primitive $l$-th root of unity and the $h_{ij} \in \mathbb{Z}$. Let $h$ be the cardinality of the image of the homomorphism 
\[
    \mathbb{Z}^n \xrightarrow{(h_{ij})} \mathbb{Z}^n \xrightarrow{\pi} \left(\mathbb{Z}/l\mathbb{Z}\right)^n,
\]
where $\pi$ denotes the canonical epimorphism. Then \[\pideg(\mathcal{O}_{\mathbf{q}}(\mathbb{K}^n))=\pideg(\mathcal{O}_{\mathbf{q}}\left(\mathbb{(K^*)}^n\right)=\sqrt{h}.\]
\end{prop}
It is well known that a skew-symmetric matrix over $\mathbb{Z}$ such as our matrix $H:=(h_{ij})$ can be brought into a $2\times 2$ block diagonal form (commonly known as skew normal form) by a unimodular matrix $W\in GL_n(\mathbb{Z})$. Then $H$ is congruent to a block diagonal matrix of the form: 
\[WHW^{t}=\diagonal\left(\begin{pmatrix}
0&h_1\\
-h_1&0
\end{pmatrix},\cdots,\begin{pmatrix}
0&h_s\\
-h_s&0
\end{pmatrix},\bf{0}_{n-2s}\right),\]
where $\bf{0}_{n-2s}$ is the square matrix of zeros of dimension equals $\dime(\kere H)$, so that $2s=\ran(H)=n-\dime(\kere H)$ and $h_i\mid h_{i+1}\in \mathbb{Z}\setminus\{0\},~~\forall \ \ 1\leq i \leq s-1$. The nonzero $h_1,h_1,\cdots,h_s,h_s$ (each occurs twice) are called the invariant factors of $H$. The following result simplifies the calculation of $h$ in the statement of proposition (\ref{quan}) by the properties of the integral matrix $H$, namely the dimension of its kernel, along with its invariant factors and the value of $l$.
\begin{lemma}\emph{(\cite[Lemma 5.7]{ar})}\label{mainpi}
Take $1\neq q\in \mathbb{K}^*$, a primitive $m$-th root of unity. Let $H$ be a skew symmetric integral matrix associated to $\mathbf{q}$ with invariant factors $h_1,h_1,\cdots,h_s,h_s$. Then PI degree of $\mathcal{O}_{\mathbf{q}}(\mathbb{K}^n)$ is given as \[\pideg(\mathcal{O}_{\mathbf{q}}(\mathbb{K}^n))=\pideg(\mathcal{O}_{\mathbf{q}}\left(\mathbb{(K^*)}^n\right)=\prod_{i=1}^{\frac{n-\dime(\kere H)}{2}}\frac{m}{\gcdi(h_i,m)}.\]
\end{lemma}
An easy corollary to Lemma \ref{mainpi} is the following
\begin{coro}\label{PI}
    $\pideg(\mathcal{O}_{\mathbf{q}}(\mathbb{K}^n))=\pideg(\mathcal{O}_{\mathbf{q}}\left(\mathbb{(K^*)}^n\right) \leq m^{\frac{\ran H}{2}}$.
\end{coro}
In \cite[Example 5.4]{han}, Hanyal proved that the following
\begin{align*}
  PI-deg(\mathcal{O}_{q}(\mathfrak{o}\mathbb{K}^{2n})=\left \{
\begin{array}{lc}
m^{n-1} \ \ \ \  \text{when} \ \  m \  \text{is odd}.\\~\\
 \frac{m^{n-1}}{2^{n-2}} \ \ \ \ \text{when} \ \ m \in 4\mathbb{Z}.\\~\\
 \frac{m^{n-1}}{2^{\lfloor\frac{n-1}{2}\rfloor}} \ \ \ \ \text{when} \ \ m \in 2\mathbb{Z} \ \text{but} \ m \notin 4\mathbb{Z}
\end{array}
\right. 
\end{align*}

\section{PI degree of prime factor algebra of quantum euclidean $2n$-space.}
Our main goal in this section is to give an explicit description of PI degree of prime factor algebra $\mathcal{O}_{q}(\mathfrak{o}\mathbb{K}^{2n})/\langle \omega_p\rangle$ for $2 \leq p \leq n$.
\begin{theo}\label{main} 
Under the roots of unity assumption for any $p$ with $2 \leq p \leq n-1$, the factor algebra $\mathcal{O}_q(\mathfrak{o}\mathbb{K}^{2n})/\langle \omega_p \rangle$ is a prime PI ring and $\pideg\mathcal{O}_q(\mathfrak{o}\mathbb{K}^{2n})/\langle \omega_p \rangle = m^{n-2}$.
\end{theo}
\begin{proof}
Fix an index $p$ such that $2 \leq p \leq n-1$. Then the factor algebra $\mathcal{O}_q(\mathfrak{o}\mathbb{K}^{2n})/\langle \omega_p\rangle$ is a domain (see \cite[Lemma 2.6]{aj}). Hence the algebra $\mathcal{O}_q(\mathfrak{o}\mathbb{K}^{2n})/\langle \omega_p \rangle$ is a prime PI affine algebra under roots of unity assumption. Let $A_p$ be the $\mathbb{K}$-algebra generated by the variables $X_1,\ldots,X_n,Y_1,\ldots, Y_n$ subject to the relations:
\[\begin{array}{ll}
X_iX_j=qX_jX_i, \ \forall i<j, \ \ Y_iY_j=q^{-1}Y_jY_i, &\forall i<j.\\
X_iY_j=q^{-1}Y_jX_i, \ \ \forall i\neq j, \ \ X_{p}Y_{p}=q^{-2}Y_{p}X_{p},\\
X_iY_i=Y_iX_i+\sum_{k<i}(1-q^{-2})Y_kX_k, &\forall i \neq p.
\end{array}\]

Note that it is easy to prove that the factor algebra $\mathcal{O}_q(\mathfrak{o}\mathbb{K}^{2n})/\langle \omega_p \rangle$ is isomorphic to the factor algebra $A_p/I_p$ where $I_p$ is the ideal of $A_p$ generated by $\sum_{k=1}^{p}(1-q^{-2})Y_kX_k$. Hence we conclude that
\begin{equation}\label{noteq}
    \pideg (\mathcal{O}_q(\mathfrak{o}\mathbb{K}^{2n})/\langle \omega_p \rangle)=\pideg{(A_p/I_p)}\leq \pideg{(A_p)}.
\end{equation} 
So now our focus shifts towards the computation of PI-degree of $A_p$. The algebra $A_p$ also has an iterative skew polynomial presentation 
\[\mathbb{K}[Y_1][X_1;\tau_1][Y_2;\sigma_2][X_2;\tau_2,\delta_2]\cdots[Y_n;\sigma_n][X_n;\tau_n,\delta_n]\] where the automorphisms $\tau_i,\sigma_i$ and derivations $\delta_i$ are defined by 
\[\begin{array}{ll}
\sigma_i(Y_j)=qY_j,\ \ \ \ \ \ \ \ \ \ \ \ \  \sigma_i(X_j)=qX_j, &j<i.\\
\tau_i(Y_j)=q^{-1}Y_j,\ \ \ \ \ \ \ \ \ \ \ \ \  \tau_i(X_j)=q^{-1}X_j, &j <i.\\
 \tau_i(Y_i)=Y_i \ \ \forall i&\\
 \delta_i(X_j)=\delta_i(Y_j)=0, &j<i.\\
 \delta_i(Y_i)=\sum _{k < i}(1-q^{-2})Y_kX_k, \ \  & \forall i \neq p+1.\\
 \delta_{p+1}(Y_{p+1})=0.
\end{array}
\]
One checks that the hypothesis 
of the derivation erasing process independent of characteristic due to by A. Leroy and J. Matczuk \cite[Theorem 7]{lm2} holds for $A_p$, so we obtain $\pideg(A_p)=\pideg(\mathcal{O}_H(\mathbb{K}^{2n}))$. 
\par Now the $2n \times 2n$ matrix $H$ is of the form 
\begin{center}
    $\begin{pmatrix}
0 & 0 & -1 & 1 &-1 & 1  &\ldots & -1 & 1 & -1 & 1 &\ldots & -1 & 1\\
0 & 0 & -1 & 1 &-1 & 1  &\ldots & -1 & 1 & -1 & 1 &\ldots & -1 & 1\\
1 & 1 & 0 & 0 &-1 & 1  &\ldots & -1 & 1 & -1 & 1 &\ldots & -1 & 1\\
-1 & -1 & 0 & 0 &-1 & 1  &\ldots & -1 & 1 & -1 & 1 &\ldots & -1 & 1\\
\vdots &\vdots &\vdots &\vdots &\vdots &\vdots &\ddots&\vdots&\vdots&\vdots&\vdots&\ddots&\vdots&\vdots\\
1 & 1 & 1 & 1 & 1 & 1  &\ldots & 0 & 2 & -1 & 1 &\ldots & -1 & 1\\
-1 & -1 & -1 & -1 & -1 & -1  &\ldots & -2 & 0 & -1 & 1 &\ldots & -1 & 1\\
1 & 1 & 1 & 1 & 1 & 1  &\ldots & 1 & 1 & 0 & 0 &\ldots & -1 & 1\\
-1 & -1 & -1 & -1 & -1 & -1  &\ldots & -1 & -1 & 0 & 0 &\ldots & -1 & 1\\
\vdots &\vdots &\vdots &\vdots  &\vdots &\vdots &\ddots&\vdots&\vdots&\vdots&\vdots&\ddots&\vdots&\vdots\\
1 & 1 & 1 & 1 & 1 & 1  &\ldots & 1 & 1 & 1 & 1 &\ldots & 0 & 0\\
-1 & -1 & -1 & -1 & -1 & -1  &\ldots & -1 & -1 & -1 & -1 &\ldots & 0 & 0
\end{pmatrix}$
\end{center}
where
\begin{itemize}
    \item the $2p-1$-th row is $(1 , 1 , 1 , 1 ,\ldots , 0 , 2, -1, 1 ,\ldots , -1 , 1)$.
    \item  the $2p$-th row is $(-1 ,- 1 , -1 , -1 ,\ldots , -2 , 0 ,-1,1,\ldots , -1 , -1)$.
    \item the $2p+1$-th row is $(1 , 1 , 1 , 1 , 1 , 1  ,\ldots , 1 , 1 , 0 , 0 ,\ldots , -1 ,1)$.
    \item The $2p+2$-th row is $(-1 , -1 , -1 , -1 , -1 , -1  ,\ldots , -1 , -1 , 0 , 0 ,\ldots , -1 ,1)$
\end{itemize}
  Now we perform the following elementary row operations:
\begin{itemize}
\item Replace row $1$ with row $1-~\text{row}~2$.
    \item Replace row $(2p-1)$ with row $(2p-1)~-~\text{row}~2p$.
    \item Replace row $(2p+1)$ with row $(2p+1)~-~\text{row}~(2p+2)$.
    \item Replace row $(2n-1)$ with row $(2n-1)~+~\text{row}~2n$
\end{itemize}
The resulting matrix has two zero rows: row $1$ and row $(2n-1)$ and two identical rows: row $(2p-1)$ and row $(2p+1)$, which are both equal to $(2,2,\ldots,2,2,0,0,\ldots,0)$ where the first $2p$ many coordinates are $2$.\\
Since the rank of a skew-symmetric matrix is even, we have $\text{rank}~H \leq 2n-4$. So from 
Corollary \ref{PI}, we have $\pideg(A_p)=\pideg(\mathcal{O}_H(\mathbb{K}^{2n})) \leq m^{n-2}$. Hence $\pideg\mathcal{O}_{q}(\mathfrak{o}\mathbb{K}^{2n})/\langle \omega_p \rangle \leq m^{n-2}$.\\
 Now we can establish the equality of the above
by constructing a simple $\mathcal{O}_{q}(\mathfrak{o}\mathbb{K}^{2n})/\langle \omega_p \rangle$-module with $\mathbb{K}$-dimension $m^{n-2}$.
\par  Let $(\underline{\alpha}, \underline{\lambda},\underline{\gamma}):=(\alpha_1,\ldots,\alpha_n,\lambda_1,\ldots,\lambda_n, \gamma_1,\gamma_2,\gamma_3)\in \mathbb{K}^{2n+3}$ be such that all the scalar except $\alpha_{p+1}, \lambda_p$ are non zero. Given such $(\underline{\alpha}, \underline{\lambda},\underline{\gamma})$, let $M(\underline{\alpha},\underline{\lambda},\underline{\gamma})$ be the $\mathbb{K}$-vector space with basis consisting of \[\{e(\underline{a})~|~\underline{a}=(a_1,\cdots,a_n),~0\leq a_i \leq m-1,~ 1\leq i\leq n~ \text{with}~ i \neq 1,p+1, a_1=a_{p+1}=0\}.\]
Let $e_{\pm i}$ denote the vector $(0,\cdots,{\pm 1},\cdots,0)$ whose $i$-th coordinate is $\pm 1$ and the remaining are zero. Define the $\mathcal{O}_{q}(\mathfrak{o}\mathbb{K}^{2n})$-module structure on the $\mathbb{K}$-space $M\left(\underline{\alpha},\underline{\lambda},\underline{\gamma}\right)$ by the action of each generator on the basis vectors as follows: 
\begin{align*}
    e(\underline{a})X_1&=\gamma_1  \prod_{\substack{j=2\\ j \neq i+1}}^{n}q^{-a_j}e(\underline{a})\\
e(\underline{a})Y_1&=\frac{\gamma_1^{-1}\lambda_1}{1-q^{-2}} \prod_{\substack{j=2\\ j \neq i+1}}^{n}q^{-a_j} e(\underline{a})\\
e(\underline{a})X_k&=\alpha_k\prod_{\substack{j=2\\ j \neq p+1}}^{k-1}q^{a_j}~e(\underline{a}+e_k) \ \ \text{for $k \neq p+1$}\\
e(\underline{a})Y_k&=\alpha_k^{-1}\prod_{\substack{j=2\\ j \neq p+1}}^{k-1}q^{-a_j}\smashoperator{\prod_{\substack{j=k+1\\ j \neq p+1}}^{n}}q^{-2a_j}\ \frac{\lambda_k-q^{-2a_k}\lambda_{k-1}}{1-q^{-2}}e(\underline{a}+e_{-k}) \ \ \text{for $k \neq p+1$}\\
e(\underline{a})X_{p+1} &=\alpha_p \prod\limits_{j=2}^{p-1}q^{a_j}\ q^{(a_p+1)}\ \gamma_2 \  (q^{-2}-1)\ \lambda_{p-1}^{-1}\  e(\underline{a}+e_p)\\
e(\underline{a})Y_{p+1}&=\alpha_p^{-1}\prod_{j=p+2}^{n}q^{-a_j}\smashoperator{\prod_{\substack{j=2 \\ j \neq p,p+1}}^{n}}q^{-a_j} q^{-(a_p-1)} \ \gamma_3 \ e(\underline{a}+e_{-p})
\end{align*}
It can be easily verified this is $\mathcal{O}_{q}(\mathfrak{o}\mathbb{K}^{2n})$-module with $M(\underline{\alpha},\underline{\lambda},\underline{\gamma})\omega_p=0$. Hence $M(\underline{\alpha},\underline{\lambda},\underline{\gamma})$ is a simple $\mathcal{O}_{q}(\mathfrak{o}\mathbb{K}^{2n})/\langle \omega_p \rangle$-module with $\mathbb{K}$-dimension $m^{n-2}$. Thus $\pideg\mathcal{O}_{q}(\mathfrak{o}\mathbb{K}^{2n})/\langle \omega_p \rangle=m^{n-2}$.
\end{proof}
\section{Maximal Dimensional Simple Modules.} 
\textbf{Notations:}
    Let $N$ be a simple module over $\mathcal{O}_{q}(\mathfrak{o}\mathbb{K}^{2n})$. Now, for each $i=1,\cdots,n$, $x_i^{m}$, $y_i^{m}$ are central elements of $\mathcal{O}_{q}(\mathfrak{o}\mathbb{K}^{2n})$ and hence they act as scalars on $N$. So, $x_i^{m}=\alpha_i$ and $y_i^{m}=\beta_i$ on $N$ for some $\alpha_i, \beta_i \in \mathbb{K}$. For each $\underline{a}=(a_1,a_2,\cdots,a_n)$, let us define 
\[b_{\underline{a}}=\prod_{i=n}^{1}b_{a,i}.\]
 where 
\[b_{a,i}=\left \{
    \begin{array}{cc}x_i^{a_i} \ \ \ \  \text{when} \ \  \alpha_i \neq 0
         &  \\
         y_i^{a_i} \ \ \ \ \text{when} \ \ \alpha_i = 0& 
    \end{array}
    \right.\]
Let us define\[ I:=\{i \in \{1,\cdots,n\}: \alpha_i=0\}, \ \
    J:=\{j \in \{1,\cdots,n\}: \beta_j=0\}\]
Also $e_i=(0,\cdots,0,1,0,\cdots,0)$ where $1$ is in the $i$-th place and all the other places are $0$ and and $e_{-i}=(0,\cdots,0,-1,0,\cdots,0)$ where $-1$ is in the $i$-th place and all the other places are $0$. Throughout this article we are going to use these notations.
\par Let us start by assuming $N$ is a $\Omega$-torsionfree simple $\mathcal{O}_{q}(\mathfrak{o}\mathbb{K}^{2n})$-module. 
Then the following two cases arise.\\
\textbf{Case I:}
$\mathbf{I = \emptyset}$. It is notable that the elements $x_1$, $\omega_i$, $i=1,\cdots,n$ in $\mathcal{O}_{q}(\mathfrak{o}\mathbb{K}^{2n})$ exhibits commutativity. So we have $v \in N$ with $v \neq 0$ such that  $v\omega_i=\lambda_iv$ for $i=1,\cdots,n$ and $vx_1=\gamma_1v$. Since $N$ is $\Omega$-torsionfree, for each $i=1,\cdots,n-1$, $\lambda_i$ and $\gamma_1$ are nonzero scalars. Let $N_1$ be the subspace of the vector space $N$ spanned by the vectors $e(\underline{a})$  where $e(\underline{a}):=vb_{\underline{a}}$ where $\underline{a}=(0,a_2,\cdots,a_n)$  where $0 \leq a_i \leq m-1$. We will show that these vectors span a submodule. In fact, after some direct calculation, we get the following results. For $i=1$, we have 
\begin{align*}
e(\underline{a})x_1
&= \gamma_1\prod_{j=2}^{n}q^{-a_j} e(\underline{a})\\
e(\underline{a})y_1&= \gamma_1^{-1}\lambda_1(1-q^{-2})^{-1}\prod_{j=2}^{n}q^{-a_j} e(\underline{a})\\
e(\underline{a})x_i&= \prod_{j=2}^{i-1}~q^{a_j}~ e(\underline{a}+e_i) \ \ \text{for $2 \leq i \leq n$}\\
e(\underline{a})y_i&= \prod_{j=2}^{i-1}q^{-a_j}\smashoperator{\prod_{j=i+1}^{n}}q^{-2a_j} ~\frac{\lambda_i-q^{-2a_i}\lambda_{i-1}}{1-q^{-2}} 
        e(\underline{a}+e_{-i})\ \ \text{for $a_i \neq 0$.}\\
e(\underline{a})y_i&=\alpha_i^{-1}\prod_{j=2}^{i-1}q^{-a_j}\smashoperator{\prod_{j=i+1}^{n}}q^{-2a_j} 
        \frac{(\lambda_i-\lambda_{i-1})}{(1-q^{-2})} 
        e(\underline{a}+e_{-i})\ \ \text{for $a_i = 0$.}
\end{align*}
Therefore, owing to simpleness of $N$, $N=N_1$. Now the following result deciphers the dimension of $N$.
\begin{theo}\label{eucrep}
$N$ has dimension $m^{n-1}$.
\end{theo}
\begin{proof}
We use the principle of induction to show that all the vectors $e(0,a_2,\cdots ,a_n)$ for $0 \leq a_i \leq m-1$ are linearly independent. Let the result be true for any such $k$ vectors. Let $S:=\{e\left(0,a_2^{(i)},\cdots,a_n^{(i)}\right):i=1,\cdots,k+1\}$ be a set of $k+1$ such vectors. Suppose  
\[ p:=\sum_{i=1}^{k+1} \zeta_i~e\left(0,a_2^{(i)},\cdots,a_n^{(i)}\right)=0 \] 
for some $\zeta_i\in \mathbb{K}$. Now, $e\left(0,a_2^{(k)},\cdots,a_n^{(k)}\right) \neq e\left(0,a_2^{(k+1)},\cdots,a_n^{(k+1)}\right)$.
We can choose the largest index $r$ such that $a_r^{(k)}\neq a_r^{(k+1)}$ in $\mathbb{Z}/{m\mathbb{Z}}$. Now the vectors $e\left(0,a_2^{(k)},\cdots,a_n^{(k)}\right)$ and $e\left(0,a_2^{(k+1)},\cdots,a_n^{(k+1)}\right)$ are eigenvectors of $x_ry_r$ associated with the eigenvalues
\[ q^{-2(a_n^{(k)}+\cdots+a_{r+1}^{(k)})}(1-q^{-2})^{-1}(\lambda_r-q^{-2a_r^{(k)}-2}\lambda_{r-1})=\nu_k(say) \]
and 
\[q^{-2(a_n^{(k+1)}+\cdots+a_{r+1}^{(k+1)})}(1-q^{-2})^{-1}(\lambda_r-q^{-2a_r^{(k+1)}-2}\lambda_{r-1})=\nu_{k+1} (say)\]
respectively. We claim that $\nu_k\neq \nu_{k+1}$. Indeed, 
\[\nu_k=\nu_{k+1} \implies q^{-2(a_r^{(k)}-a_r^{(k+1)})}=1 \implies m\divides \left(a_r^{(k)}-a_r^{(k+1)}\right),\]
which is a contradiction.
\par Now 
\[\begin{array}{ll}
    0&=px_ry_r-\nu_{k+1}p\\
    &=\sum_{i=1}^{k} \zeta_i(\nu_i-\nu_{k+1})~e\left(0,a_2^{(i)},\cdots,a_n^{(i)}\right)
\end{array}\]
Using induction hypothesis,
\begin{center}
    $\zeta_i(\nu_i-\nu_{k+1})=0$ for all $i=1,\cdots,k$.
\end{center}
As $\nu_k \neq \nu_{k+1}$, $\zeta_k=0$.
Putting $\zeta_k=0$ in (4.1) and using induction hypothesis we have $\zeta_i=0$ for all $i=1,\cdots,k+1$.\\
Hence the result.
\end{proof}
Now the question is whether there really exists such a module. The answer is yes. Actually, one can easily verify the module defining rules. In order to do that we need the following checks for $1\leq i\leq n$ and $0\leq a_i\leq {m-1}$,
\begin{align}
  &e(\underline{a})y_iy_j=q^{-1}~e(\underline{a})y_jy_i,\ \ \ \forall i<j\label{w1}\\
  &e(\underline{a})x_iy_j=q^{-1}e(\underline{a})y_jx_i,\ \ \ \ \ \forall \  \ \ i \neq j.\label{w2}\\
  &e(\underline{a})x_ix_j=qe(\underline{a})x_jx_i, \ \ \forall \ \ \ i<j.\label{w3}\\
   &e(\underline{a})x_{i}y_{i}=e(\underline{a})(y_{i}x_{i}+\sum_{l<i}(1-q^{-2})y_lx_l),\ \ \ \ \ \forall \ i=1,\cdots,n.\label{w4}
\end{align}
The relations (\ref{w1}) - (\ref{w3}) are easy to show.
For relation (\ref{w4}) we have the following calculation. As for $i=1$ the calculations are trivial, we only compute for $2 \leq i \leq n$
\[e(\underline{a})x_{i}y_{i}
=q^{-2(a_n+\cdots+a_{i+1})}\left(1-q^{-2}\right)^{-1}\left(\lambda_i-q^{-2a_i-2}\lambda_{i-1}\right)~e(\underline{a})\]

Now for the right hand side we have,
\begin{align*}
&e(\underline{a})\left(y_{i}x_{i}+\sum_{l<i}(1-q^{-2})y_lx_l\right)\\
&=e(\underline{a})y_{i}x_{i}+e(a_2,\cdots,a_n)\left(\sum_{l<i}(1-q^{-2})y_lx_l\right)\\
&=e(\underline{a})y_{i}x_{i}+e(\underline{a})\omega_{i-1}\\
&=\left[q^{-2(a_n+\cdots+a_{i+1})}\left(1-q^{-2}\right)^{-1}(\lambda_i-q^{-2a_i}\lambda_{i-1})+\lambda_{i-1}~q^{-2(a_n+\cdots+a_i)}\right]~e(\underline{a})\\
&=q^{-2(a_n+\cdots+a_{i+1})}\left(1-q^{-2}\right)^{-1}\left(\lambda_i-q^{-2a_i}\lambda_{i-1}+q^{-2a_i}\lambda_{i-1}-q^{-2a_i-2}\lambda_{i-1}\right)~e(\underline{a})\\
&=q^{-2(a_n+\cdots+a_{i+1})}\left(1-q^{-2}\right)^{-1}\left(\lambda_i-q^{-2a_i-2}\lambda_{i-1}\right)~e(\underline{a}).
\end{align*} 
\textbf{Case II:}
Suppose $\mathbf{I=\{i_1,\cdots,i_p\} (\neq \emptyset)}$. Since for each $i \in I$, $x_i$ is nilpotent on $N$, we must have $\cap_{k=1}^{p}\text{Ker}~~ x_{i_k} \neq \emptyset$. Now each $\omega_i$ and $x_1$ must keep $\cap_{\substack{k=1}}^{p}\text{Ker}~~ x_{i_k}$ invariant. So we can take the common eigenvector $v \in \cap_{k=1}^{p}\text{Ker}~~ x_{i_k}$ of $\omega_i$ and $x_1$. Let $v\omega_i=\lambda_iv$ for $i=1,\cdots,n$ and $vx_1=\gamma_1v$. As $N$ is $\Omega$-torsionfree, we have for each $i=1,\cdots,n-1$, $\lambda_i$ and $\gamma_1$ are nonzero scalars.
\par Then we have the following lemma
\begin{lemma}\label{pp}
     For $1 \leq i \leq n$, let
       \begin{align*}
           z_i&=
    \begin{dcases}
         x_i \ \ \text{if}\ \ i \notin I\\
         y_i \ \ \text{if}\ \ i \in I
    \end{dcases}
       \end{align*}
       Then for all $1 \leq i \leq n$ and for any $a_i$ with $0 \leq a_i \leq m-1$, we have $vz_i^{a_i} \neq 0$.
\end{lemma}
\begin{proof}
     Let us start by fixing $i$, $1 \leq i \leq n$. Then we have the following cases\\
       \textbf{Case 1:} If $i \notin I$, then as $x_i^m=\alpha_i \neq 0$, we have $vx_i^{a_i} \neq 0$ for all $a_i$ with $0 \leq a_i \leq m-1$.\\
       \textbf{Case 2:} If $i \in I \setminus J$, then as $y_i^m=\beta_i \neq 0$, we have $vy_i^{a_i} \neq 0$ for all $a_i$ with $0 \leq a_i \leq m-1$.\\
       \textbf{Case 3:} Note that for $i \in I \cap J$, $vx_{i}=0$ and $y_i^m=0$ on $N$. Suppose $s$, $1 \leq s \leq m$ be the maximum index such that $vy_{i}^{s-1}\neq 0$. If $i=1$, we have
       $$0=vy_1^sx_1=vy_1^{s-1}y_1x_1={(1-q^{-2})}^{-1}vy_1^{s-1}\omega_1,$$ but this contradicts the assumption that $N$ is $\Omega$-torsion-free. Hence $i >1$.
       Note,
\begin{align*}
0=vy_{i}^rx_{i}=&v\left(x_{i}y_{i}^s-\frac{1-q^{2s}}{1-q^2}\omega_{(i-1)}y_{i}^{s-1}\right)\\
    =&-\frac{1-q^{2s}}{1-q^2}\lambda_{(i-1)}vy_{i}^{s-1}. 
\end{align*}
\par Since $vy_{i}^{s-1},\lambda_{i-1} \neq 0$, we must have $s=m$.  Here also for all $a_i=1,2,\ldots,m-1$, $vy_i^{a_i} \neq 0$.\\
So finally we have proved that for any $i$, $vz_i^{a_i}\neq 0$ for all $1 \leq l \leq l-1$.
\end{proof}
Consider the vector subspace $N_1$ of $N$ generated by the vectors $e(\underline{a})$  where $e(\underline{a}):=vb_{\underline{a}}$ where $\underline{a}=(0,a_2,\cdots,a_n)$ where $0 \leq a_i \leq m-1$.
From Lemma \ref{pp}, we have all $e(\underline{a}) \neq 0$.
 We will show that these vectors spans a submodule. In fact, after some direct calculation as done in Case-I, we get 
\begin{align*}
\intertext{For $i=1$,}
e(\underline{a})x_1
&= \gamma_1\prod_{\substack{j=2,\\j \not\in I}}^{n} q^{-a_j} 
   \prod_{\substack{j=2,\\ j \in I}}^{n} q^{a_j} e(\underline{a}) \\
e(\underline{a})y_1
&= \gamma_1^{-1}\frac{\lambda_1}{(1-q^{-2})}
   \prod_{\substack{j=2,\\ j \not\in I}}^{n} q^{-a_j} 
   \prod_{\substack{j=2,\\ j \in I}}^{n} q^{a_j} e(\underline{a}) \\
\intertext{For $i \not\in I$,}
e(\underline{a})x_i
&= \prod_{j=2}^{i-1}q^{a_j}e(\underline{a}+e_i).\\
e(\underline{a})y_i
&=
\begin{dcases}
   \prod_{j=2}^{i-1}q^{-a_j} 
   \smashoperator{\prod_{\substack{j=i+1,\\ j \not\in I}}^{n}} q^{-2a_j} 
   \smashoperator{\prod_{\substack{j=i+1,\\ j \in I}}^{n}}
   q^{2a_j} \frac{\lambda_{i}- q^{-2a_{i}}\lambda_{i-1}}{1-q^{-2}}
   e(\underline{a}+e_{-i}) 
   & \text{if $a_i>0$.} \\
   \alpha_i^{-1} \prod_{j=2}^{i-1}q^{-a_j}  
   \smashoperator{\prod_{\substack{j=i+1,\\j \not\in I}}^{n}}q^{-2a_j}
   \smashoperator{\prod_{\substack{j=i+1,\\ j \in I}}^{n}} 
   q^{2a_j}\frac{\lambda_{i}-\lambda_{i-1}}{1-q^{-2}}e(\underline{a}+e_{-i}) 
   & \text{if $a_i=0$}. 
\end{dcases} \\
\intertext{For $i \in I$,}
e(\underline{a})x_{i}
&=  
\begin{dcases}
    \prod_{j=2}^{i-1}q^{a_j} 
    \smashoperator{\prod_{\substack{j=i+1,\\j \not\in I}}^{n}} q^{-2a_j}
    \smashoperator{\prod_{\substack{j=i+1,\\ j \in I}}^{n}}q^{2a_j} \ 
    \frac{1- q^{2a_{i}}}{q^{2}-1}\lambda_{(i-1)} \ e(\underline{a}+e_{-i})
    & \text{if $a_{i}\neq 0$.} \\
    0 & \text{if $a_{i}=0$.}  
\end{dcases}\\
e(\underline{a})y_{i} 
&= \prod_{j=2}^{i-1}q^{-a_j} e(\underline{a}+e_i).
\end{align*} 
Therefore, again owing to simpleness of $N$, $N=N_1$. Now the following result deciphers the dimension of $N$.
\begin{theo}
$N$ has dimension $m^{n-1}$.
\end{theo}
\begin{proof}
Note that for $e(0,a_2,\cdots,a_n) \neq e(0,b_2,\cdots,b_n)$, both of these vectors are eigenvectors of some of the operator $x_iy_i$ for $i=2,\cdots,n$ with distinct eigenvalues. Using the same method as in Theorem \ref{eucrep}, we can obtain the desired result.
\end{proof}
  As in Case I, we can follow the same method to ensure there does exist a module defined above. So we get the following result
\begin{theo}
Every $\Omega$-torsionfree simple $\mathcal{O}_{q}(\mathfrak{o}\mathbb{K}^{2n})$-module is maximal dimensional.
\end{theo}
\begin{theo}
Let $N$ be a $\omega_2,\omega_3,\cdots,\omega_{n-1}$ torsionfree simple $\mathcal{O}_{q}(\mathfrak{o}\mathbb{K}^{2n})$-module which is $\omega_1$ torsion i.e. $\omega_1N=0$. Then $N$ is a maximal dimensional simple $\mathcal{O}_{q}(\mathfrak{o}\mathbb{K}^{2n})$-module.
\end{theo} 
\begin{proof}
Since $N$ is a simple $\omega_1$-torsion $\mathcal{O}_{q}(\mathfrak{o}\mathbb{K}^{2n})$-module, $\omega_1=(1-q^{-2})y_1x_1=0$ on $N$. Since $x_1$ and $y_1$ are normal in $\mathcal{O}_{q}(\mathfrak{o}\mathbb{K}^{2n})$, they are either $0$ or invertible on $N$. We eliminate the case where $x_1$ and $y_1$ are both $0$ on $N$ otherwise $N$ becomes a simple module over $\mathcal{O}_{q}(\mathfrak{o}\mathbb{K}^{2n-2})$.
\par Without loss of generality let us assume $x_1$ is invertible and $y_1=0$ on $N$. Note that as $\omega_1=0$ on $N$, $N$ becomes a simple module over $B:=\mathcal{O}_{q}(\mathfrak{o}\mathbb{K}^{2n})/\langle \omega_1 \rangle$ and $x_2$ and $y_2$ are normal in $B$. Since $\omega_2 \neq 0$ on $N$, both $x_2$ and $y_2$ must be invertible on $N$. Then the following two cases arise.\\
\textbf{Case-I:} $I = \emptyset$. Consider the commutating operators $\omega_i, \ i=2,3, \cdots,n$, $x_1y_2, \ x_2y_2$. Then we have $v \in N$ with $v \neq 0$ such that \[v\omega_i=\lambda_iv, \ vx_1y_2=\mu_1v, \ vx_2y_2=\mu_2v.\] Note that $\lambda_2,\cdots,\lambda_{n-1},\mu_1,\mu_2 \neq 0$. The subspace $N_1$ of $N$ spanned by the vectors $e(\underline{a}):=vb_{\underline{a}}$ where $a=(a_1,0,a_3,\cdots,a_n)$ and $0 \leq a_i \leq m-1$ becomes a submodule. The actions of the generators are as follows
\begin{align*}
e(\underline{a})x_1&=e(\underline{a}+e_1)\\
e(\underline{a})y_1&= 0.\\
e(\underline{a})y_2&=q^{1-a_1} \prod_{j=3}^{n}q^{-2a_j}~\mu_1~e(\underline{a}+e_{-1})\\
e(\underline{a})x_2&=\mu_1^{-1}\mu_3 q^{-a_1} e(\underline{a}+e_1)\\
e(\underline{a})x_i&= \prod_{\substack{j=1\\j \neq 2}}^{i-1}q^{a_j}e(\underline{a}+e_i) \ \ \text{for $3 \leq i \leq n$.}\\
e(\underline{a})y_i &=
\begin{dcases}\prod_{\substack{j=1\\j \neq 2}}^{i-1}q^{-a_j}\smashoperator{\prod_{j=i+1}^{n}}q^{-2a_j} \frac{(\lambda_i-q^{-2a_i}\lambda_{i-1})}{(1-q^{-2})}e(\underline{a}+e_{-i}),&\text{when} \ \ a_i>0  \\
\alpha_i^{-1}\prod_{\substack{j=1\\j \neq 2}}^{i-1}q^{-a_j}\smashoperator{\prod_{j=i+1}^{n}}q^{-2a_j}\frac{(\lambda_i-\lambda_{i-1})}{(1-q^{-2})}e(\underline{a}+e_{-i}),& \text{when} \ \ a_i=0
\end{dcases}
\end{align*}
Since $N$ is simple, $N=N_1$.
\par $N$ has dimension $m^{n-1}$. The calculations are similar to Theorem \ref{eucrep}. Note that for $e(a_1,0,a_3,\cdots,a_n) \neq e(b_1,0,b_3,\cdots,b_n)$, if $a_1 \neq b_1$ and $a_3=b_3,\cdots,a_n=b_n$, then 
\[e(a_1,0,a_3,\cdots,a_n)x_1y_2=q^{1-a_1}e(a_1,0,a_3,\cdots,a_n)\] and
\[e(b_1,0,b_3,\cdots,b_n)x_1y_2=q^{1-b_1}e(b_1,0,b_3,\cdots,b_n)\]
and the eigenvalues $q^{1-a_1}$ and $q^{1-b_1}$ are distinct.\\
\textbf{Case II:}
 Suppose $\mathbf{I} = \{i_1,\cdots,i_p\}$ where $I$ is not empty. From what we have discussed above, we know that $1 \notin I$. Since for each $i \in I$, $x_i$ is nilpotent on $N$, we must have $\cap_{k=1}^{p}\text{Ker}~~ x_{i_k} \neq \emptyset$. Now each $\omega_i$, $x_1y_2$, and $x_2y_2$ must keep $\cap_{k=1}^{p}\text{Ker}~~ x_{i_k}$ invariant. Thus, we can find a common eigenvector $v \in \cap_{k=1}^{p}\text{Ker}~~ x_{i_k}$ of $\omega_i$, $x_1y_2$, and $x_2y_2$.

Let $v\omega_i=\lambda_iv, \ vx_1y_2=\mu_1v, \ vx_2y_2=\mu_2v$, where $\lambda_2,\cdots,\lambda_{n-1},\mu_1,\mu_2 \neq 0$.

Let $N_1$ be the vector subspace of $N$ spanned by the vectors $e(\underline{a}):=vb_{\underline{a}}$, where $a=(a_1,0,a_3,\cdots,a_n)$ and $0 \leq a_i \leq m-1$. Using similar arguments as in Lemma \ref{pp}, we can conclude that $e(\underline{a}) \neq 0$. We also have the following
\begin{align*}
e(\underline{a})x_1&=e(\underline{a}+e_1)\\
e(\underline{a})x_2&=\mu_1^{-1}\mu_3 q^{a_1} e(\underline{a}+e_1)\\
e(\underline{a})y_1&=  0\\
e(\underline{a})y_2&= q^{1-a_1} \prod_{\substack{j=3,\\j \notin I}}^{n}q^{-2a_j}\smashoperator{\prod_{\substack{j=3,\\j \in I}}^{n}}q^{2a_j}\mu_1e(\underline{a}+e_{-1})\\
\intertext{For $i \not\in I$,} 
\ e(\underline{a})x_i& = \prod_{\substack{j=1,\\j \neq 2}}^{i-1}q^{a_j}e(\underline{a}+e_i).\\
e(\underline{a})y_i&= \begin{dcases}
     \prod_{\substack{j=1,\\j \neq 2}}^{i-1}q^{-a_j}\smashoperator{\prod_{\substack{j=i+1,\\j \not\in I}}^{n}}q^{-2a_j}\smashoperator{\prod_{\substack{j=i+1,\\ j \in I}}^{n}}q^{2a_j} \frac{\lambda_{i}-q^{-2a_{i}}\lambda_{i-1}}{1-q^{-2}}e(\underline{a}+e_{-i})& \text{when} \ \  a_i>0.  \\
   \alpha_i^{-1} \prod_{\substack{j=1,\\j \neq 2}}^{i-1}q^{-a_j} \smashoperator{\prod_{\substack{j=i+1,\\j \not\in I}}^{n}}q^{-2a_j}  \smashoperator{\prod_{\substack{j=i+1,\\ j \in I}}^{n}}q^{2a_j}\frac{\lambda_{i}-\lambda_{i-1}}{1-q^{-2}}e(\underline{a}+e_{-i}) & \text{when}\ \  a_i=0. 
\end{dcases}
\intertext{For  $i \in I$,}
e(\underline{a})x_{i}&= \begin{dcases}
     \prod_{\substack{j=1,\\j \neq 2}}^{i-1}q^{a_j}\smashoperator{\prod_{\substack{j=i+1,\\j \not\in I}}^{n}}q^{-2a_j} \smashoperator{\prod_{\substack{j=i+1,\\ j \in I}}^{n}}q^{2a_j} \ \frac{1-q^{2a_{i}}}{q^{2}-1}\lambda_{(i-1)} e(\underline{a}+e_{-i})& \text{when}\ \ a_{i}\neq 0.   \\
  0 & \text{when} \ a_{i}=0. 
\end{dcases}\\
e(\underline{a})y_{i}&= \prod_{\substack{j=1,\\j \neq 2}}^{i-1}q^{-a_j}e(\underline{a}+e_i)
\end{align*}
So $N_1=N$ and using similar techniques as in Theorem \ref{eucrep} we can conclude that $\dime N=m^{n-1}$
\end{proof}
We can summarize the section by the following theorem 
\begin{theo}\label{nec}
Let $N$ be a $\omega_2,\cdots,\omega_{n-1}$ torsion-free simple $\mathcal{O}_{q}(\mathfrak{o}\mathbb{K}^{2n})$-module. Then $N$ is maximal dimensional simple module.
\end{theo}
\begin{theo}\label{suff}
Let $N$ be a $\omega_2,\cdots,\omega_{n-1}$ torsion simple $\mathcal{O}_{q}(\mathfrak{o}\mathbb{K}^{2n})$-module. Then $N$ is not maximal dimensional. 
\end{theo}
\begin{proof}
    Since $N$ is $\omega_2,\cdots,\omega_{n-1}$ torsion simple $\mathcal{O}_{q}(\mathfrak{o}\mathbb{K}^{2n})$-module there exists some $p$ with $2 \leq p \leq n-1$ such that $N\omega _p=0$.
Then $N$ becomes $\mathcal{O}_q(\mathfrak{o}\mathbb{K}^{2n})/\langle\omega_p\rangle$-module. Since $\mathcal{O}_q(\mathfrak{o}\mathbb{K}^{2n})/\langle\omega_p\rangle$ is a prime PI ring, it follows from Proposition \ref{sim}, 
\[\dime_{\mathbb{K}}N \leq \pideg\mathcal{O}_q(\mathfrak{o}\mathbb{K}^{2n})/\langle \omega_p \rangle = m^{n-2} < m^{n-1}\]
Hence the result follows.
\end{proof}
\begin{theo}\label{maini}
    $N$ is a maximal dimensional simple $\mathcal{O}_q(\mathfrak{o}\mathbb{K}^{2n})$-module if and only if $N$ is $\omega_2,\cdots,\omega_{n-1}$ torsion-free simple $\mathcal{O}_{q}(\mathfrak{o}\mathbb{K}^{2n})$-module
\end{theo}
\begin{proof}
    Follows from Theorem \ref{nec} and Theorem \ref{suff}.
\end{proof}
\section{Center of $\mathcal{O}_{q}(\mathfrak{o}\mathbb{K}^{2n})$}
In this section, we compute the center of $\mathcal{O}_{q}(\mathfrak{o}\mathbb{K}^{2n})$ for $n\geq 2$.
\begin{theo}\label{cen}
The center $Z_n$ is the subalgebra generated by $x_1^ay_1^b$ where $a,b \in \{1,\cdots,m-1\}$ such that $a+b=m$, $x_i^m,y_i^m,i=1,\cdots,m$.
\end{theo}
\begin{proof}
 Let us use induction on $n=$ number of generators.
 \par For $n=2$, let $z=\sum_{\underline{a},\underline{b}}c_{\underline{a},\underline{b}}x_1^{a_1}y_1^{b_1}x_2^{a_2}y_2^{b_2}$ be an element of the center, where $\underline{a}=(a_1,a_2), \underline{b}=(b_1,b_2)$ are tuples of non-negative integers. We will show that if $c_{\underline{a},\underline{b}} \neq 0$ then $m \divides a_1+b_1,a_2,b_2$
 \par Now 
 \begin{align*}
 zx_2&=\left(\sum_{\underline{a},\underline{b}}c_{\underline{a},\underline{b}}x_1^{a_1}y_1^{b_1}x_2^{a_2}y_2^{b_2}\right)x_2\\
 &=\sum_{\underline{a},\underline{b}}c_{\underline{a},\underline{b}}x_1^{a_1}y_1^{b_1}x_2^{a_2}\left[x_2y_2^{b_2}-\frac{1-q^{2b_2}}{1-q^2}\omega_1y_2^{b_2-1}\right]\\
 &=\sum_{\underline{a},\underline{b}}c_{\underline{a},\underline{b}}x_1^{a_1}y_1^{b_1}x_2^{(a_2+1)}y_2^{b_2}-\sum_{\underline{a},\underline{b}}q^{-2a_2}\frac{1-q^{2b_2}}{1-q^2}c_{\underline{a},\underline{b}}\omega_1x_1^{a_1}y_1^{b_1}x_2^{a_2}y_2^{(b_2-1)}
 \end{align*}
Again 
\begin{align*}
    x_2z&=x_2\left(\sum_{\underline{a},\underline{b}}c_{\underline{a},\underline{b}}x_1^{a_1}y_1^{b_1}x_2^{a_2}y_2^{b_2}\right)\\
    &=\sum_{\underline{a},\underline{b}}q^{_(a_1+b_1))}c_{\underline{a},\underline{b}}x_1^{a_1}y_1^{b_1}x_2^{(a_2+1)}y_2^{b_2}
\end{align*}
Since $z$ is central we must have 
\begin{align*}
    q^{-(a_1+b_1)}c_{\underline{a}-e_2,\underline{b}}&=c_{\underline{a}-e_2,\underline{b}}-\frac{1-q^{2(b_2+1)}}{1-q^2}q^{-2a_2}c_{\underline{a},\underline{b}+e_2}\omega_1\\
    \left[q^{-(a_1+b_1)}-1\right]c_{\underline{a}-e_2,\underline{b}}&=-\frac{1-q^{2(b_2+1)}}{1-q^2}q^{-2a_2}c_{\underline{a},\underline{b}+e_2}\omega_1
\text{ for $a_2 \geq 1, b_2 \geq 0$}\\
\left[q^{-(a_1+b_1)}-1\right]c_{\underline{a},\underline{b}}&=-\frac{1-q^{2(b_2+1)}}{1-q^2}q^{-2(a_2+1)}c_{\underline{a}+e_2,\underline{b}+e_2}\omega_1
\text{ for $a_2,b_2 \geq 0$}
\end{align*}
Going on, we get 
\begin{align*}
    {\left[q^{-(a_1+b_1)}-1\right]}^ic_{\underline{a},\underline{b}}&={(-1)}^i\prod_{k=1}^{i}q^{-2(a_2+k)}\prod_{k=1}^{i}\frac{(1-q^{2(b_2+k)})}{{(1-q^{2})}}c_{\underline{a}+ie_2,\underline{b}+ie_2}{\omega_1}^i
\end{align*}
Now there exists $i$ such that $m \divides b_2+i$. So if $m$ does not divide $a_1+b_1$, then $c_{\underline{a},\underline{b}}=0$. If $c_{\underline{a},\underline{b}} \neq 0$, then $m \divides a_1+b_1$.
\par As $m \divides a_1+b_1$, $c_{\underline{a},\underline{b}} \neq 0$. So we must have $m \divides b_2$ from
\begin{align*}
    \left[q^{-(a_1+b_1)}-1\right]c_{\underline{a}-e_2,\underline{b}-e_2}&=-\frac{1-q^{2b_2}}{1-q^2}q^{-2a_2}c_{\underline{a},\underline{b}}~\omega_1
\text{ for $a_2,b_2 \geq 1$}
\end{align*}
\par Now $z\omega_1=\omega_1z$ gives 
\begin{align*}
\sum_{\underline{a},\underline{b}}q^{_(a_1+b_1)}c_{\underline{a},\underline{b}}\omega_1x_1^{a_1}y_1^{b_1}x_2^{a_2}y_2^{b_2}&=\sum_{\underline{a},\underline{b}}q^{(2a_2-2b_2)}c_{\underline{a},\underline{b}}\omega_1x_1^{a_1}y_1^{b_1}x_2^{a_2}y_2^{b_2}
\end{align*}
Since $m \divides b_2$, we have $m \divides a_2$. Hence the result for $n=2$.\\
Let us now assume the result holds for $N-1$. Let $\sum_{p,l}z_{p,l}x_N^{p}y_N^{l}$ be a central element in $\mathcal{O}_{q}(\mathfrak{o}\mathbb{K}^{2n})$ where $z_{p,l}$ is in $\mathcal{O}_{q}(\mathfrak{o}\mathbb{K}^{2N-2})$.\\
Now $x_Nz=zx_N$ gives
\begin{align*}
    \prod_{i=1}^{N-1}q^{-(p_i+l_i)}z_{p-1,l}&=z_{p-1,l}-\frac{1-q^{2l}}{1-q^2}q^{-2p}z_{p,l+1}~\omega_{N-1}.
\end{align*}
Finally as previously we get 
\begin{align*}
    {\left[\prod_{i=1}^{N-1}q^{-(p_i+l_i)}-1\right]}^i~z_{p,l}&={(-1)}^i\prod_{k=1}^{i}q^{-2(p+k)}\frac{1-q^{2(l+k)}}{1-q^2}z_{p+i,l+i}~\omega_{N-1}^{i}.
\end{align*}
By similar arguments as above $z_{p,l}\neq 0$ implies $m \divides \sum_{i=1}^{N-1}(p_i+l_i)$. From this it follows $m\divides l$.\\
Again $z\omega_{N-1}=\omega_{N-1}z$ will imply $m \divides p$.\\
Now as $m \divides p,l$ for any $i=1,2,\cdots, N-1$, $zx_i=x_iz$ implies $z_{p,l}~x_i=x_i~z_{p,l}$. Similarly for any $i=1,2,\cdots, N-1$, $zy_i=y_iz$ implies $z_{p,l}~y_i=y_i~z_{p,l}$. Hence $z_{p,l}$ lies in the center of $\mathcal{O}_{q}(\mathfrak{o}\mathbb{K}^{2N-2})$. Using induction hypothesis, we get the result.
\end{proof}

\begin{theo}\label{vimp}
    The following formulas hold in $\mathcal{O}_{q}(\mathfrak{o}\mathbb{K}^{2n})$:\\
    \begin{align*}
        \omega_i^m&={\left(1-q^{-2}\right)}^m\sum_{l \leq i}y_l^mx_l^m.
    \end{align*}
  for each $i=1,\cdots,n$.
\end{theo}
\begin{proof}
    Let us use induction on $n$.\\
    If $n=1$, then $\omega_i={(1-q^{-2})}y_1x_1$ and $y_1$ and $x_1$ commute with each other implying $\omega_1^m={(1-q^{-2})}^my_1^mx_1^m$.\\
    Assume the result holds for $N-1$.\\
    Now 
    \begin{align*}
        \omega_N&=\omega_{N-1}+{(1-q^{-2})}y_Nx_N.\\
         \omega_N^m&={\left[\omega_{N-1}+{(1-q^{-2})}y_Nx_N\right]}^m.\\
         &=\omega_{N-1}^m+\sum_{l=1}^{m-1}c(l,m)y_N^lx_N^l+{(1-q^{-2})}^my_N^mx_n^m.
         \end{align*} where $c(l,m) \in \mathcal{O}_{q}(\mathfrak{o}\mathbb{K}^{2N-2})$ for each $l$. The above result follows from the fact 
 \begin{itemize}
     \item $\omega_{N-1}$ quasicommutes with both $x_N$ and $y_N$.
     \item For any $p \in \mathbb{N}$, we have 
     \begin{center}
$y_N^px_N^py_Nx_N=y_N^{p+1}x_N^{p+1}+\frac{1-q^{-2p}}{1-q^{-2}}~q^{2p}\omega_{N-1}y_N^px_N^p$
     \end{center}
 \end{itemize}
         So we have
         \begin{align*}
             \omega_N^m-\omega_{N-1}^m-{(1-q^{-2})}^my_N^mx_n^m&=\sum_{l=1}^{m-1}c(l,m)y_N^lx_N^l \in  Z(\mathcal{O}_{q}(\mathfrak{o}\mathbb{K}^{2n}))
         \end{align*}
This implies $c(l,m)=0$ for each $l=1,\cdots,m-1$. Hence the result follows from induction.
\end{proof}
\section{Azumaya Locus}
 Let $\Psi_{\mathbf{m}}:Z_n \rightarrow \mathbb{K}$ denote the central character of $\mathbf{m} \in\mspect ~Z_n\left(\mathcal{O}_{q}(\mathfrak{o}\mathbb{K}^{2n})\right)$. Suppose that $\Psi_{\mathbf{m}}(x_i^m)=\alpha_i$ and $\Psi_{\mathbf{m}}(y_i^m)=\beta_i$.  From Theorem \ref{maini}, we conclude $N$ is a  maximal dimensional $\mathcal{O}_{q}(\mathfrak{o}\mathbb{K}^{2n})$ simple module if and only if $\omega_2^m,\cdots,\omega_{n-1}^m$ acts as non-zero scalars on $N$. From Theorem \ref{vimp}, we end up getting the final result.

 Azumaya locus of $\mathcal{O}_{q}(\mathfrak{o}\mathbb{K}^{2n})$ is given by
\begin{align*}
\mathcal{AL}(\mathfrak{o}\mathbb{K}^{2n})&=\{ \mathbf{m} \in  \text{Maxspec}~Z_n\left(\mathcal{O}_{q}(\mathfrak{o}\mathbb{K}^{2n})\right) :\Psi_{\mathbf{m}}(\omega_i^m) \neq 0\}\\ 
&=\{ \mathbf{m} \in \text{ Maxspec}~Z_n(\mathcal{O}_{q}(\mathfrak{o}\mathbb{K}^{2n})) : \sum_{l \leq i}\beta_l \alpha_l\neq 0,\ \ 2\leq i\leq n-1\}.
\end{align*}
\section*{Declaration Of Interest}
The author has no conflict of interest to disclose.

\end{document}